\documentclass[11pt]{article}
 \usepackage[top=1in,bottom=1in,left=1.5in,right=1.5in]{geometry}
  \usepackage{amsmath,amssymb}
  \usepackage{latexsym}
  \usepackage[dvips]{pstricks} 
   \usepackage{pst-node}
  \usepackage[dvips]{graphicx}

  \newcommand{\C}{\mathbb{C}}
  
  \newcommand{\N}{\mathbb{N}}
  \renewcommand{\P}{\mathbb{P}}
  \newcommand{\R}{\mathbb{R}}
  \newcommand{\Z}{\mathbb{Z}}

  \renewcommand{\j}{\mathbf{j}}
  
  \newcommand{\e}{\mathbf{e}}

  \newcommand{\bF}{\mathbf{F}}
  \newcommand{\bG}{\mathbf{G}}
  \newcommand{\bP}{\mathbf{P}}

  \renewcommand{\u}{\mathbf{u}}
  \renewcommand{\v}{\mathbf{v}}

  \newcommand{\x}{\mathbf{x}}
  
  \newcommand{\y}{\mathbf{y}}
  
  \newcommand{\z}{\mathbf{z}}
  \newcommand{\0}{\mathbf{0}}
  \newcommand{\1}{\mathbf{1}}

  \newcommand{\xib}{\mbox{\boldmath{$\xi$}}}
  \newcommand{\etab}{\mbox{\boldmath{$\eta$}}}

  \newcommand{\cC}{\mathcal{C}}

  \newcommand{\cF}{\mathcal{F}}
  \newcommand{\cG}{\mathcal{G}}

  \newcommand{\rD}{\mathrm{D}}

  \def\diag{\mathop{{\rm diag}}\nolimits}
  \newcommand{\hs}{\hspace*{\parindent}}
  \newcommand{\proof}{\hs \textbf{Proof.\ }}

  \newcommand{\trans}{^\top}
  \newcommand{\qed}{\hspace*{\fill} $\Box$\\}

  \newcommand{\dist}{\mathrm{dist}}

  \newcommand{\rC}{\mathrm{C}}

  \newcommand{\rS}{\mathrm{S}}

\def\setint[#1]{[#1]}
\newcommand{\cw}{\text{cw}}
\newcommand{\cwl}{\text{cw}^-}
\newcommand{\interior}{\operatorname{int}}
\newcommand{\exceptvar}[1]{_{\scriptstyle i_1\in\setint[m_1], \ldots, i_{#1-1}\in \setint[m_{#1-1}]\atop\scriptstyle  i_{#1+1}\in \setint[m_{#1+1}],\ldots, i_d\in \setint[m_d]}}
  \newtheorem{theo}{\bfseries \hs Theorem}[section]

  \newtheorem{lemma}[theo]{\bfseries \hs Lemma}
  \newtheorem{corol}[theo]{\bfseries \hs Corollary}

  \numberwithin{equation}{section} 

\begin{document}

 \title{Perron-Frobenius theorem for \\ nonnegative multilinear forms and extensions}
 \author{
 S. Friedland\footnotemark[1],\;
 S. Gaubert\footnotemark[2]
 and\;
 L. Han\footnotemark[4]
 }
 \renewcommand{\thefootnote}{\fnsymbol{footnote}}
 \footnotetext[1]{
 Department of Mathematics, Statistics and Computer Science,
 University of Illinois at Chicago, Chicago, Illinois 60607-7045,
 USA, \texttt{friedlan@uic.edu},
 }
 \footnotetext[2]{INRIA and CMAP, \'Ecole Polytechnique, 91128 Palaiseau C\'edex, \texttt{stephane.gaubert@inria.fr}}
 \footnotetext[4]{Mathematics Department, University of Michigan-Flint,
 Flint, MI 48502, \texttt{lxhan@umflint.edu} }
\footnotetext{This work was initiated during the workshop on Nonnegative matrix theory at the American Institute of Mathematics (AIM) in Palo Alto, December 1 - 5, 2008, and was partially supported by AIM.}
 \renewcommand{\thefootnote}{\arabic{footnote}}
 \date{November 20, 2010 }
 \maketitle
 \begin{abstract}
We prove an analog of Perron-Frobenius theorem
 for multilinear forms with nonnegative coefficients, and more
generally,
for polynomial maps with nonnegative coefficients.
We determine the geometric convergence rate of the power algorithm
to the unique
normalized eigenvector.
 \end{abstract}


 \section{Introduction}\label{intro}

 Let $f:\R^{m_1}\times\dots\times\R^{m_d}\to\R$ be a
 multilinear form:
 \begin{eqnarray}\label{defdmulfrm}
 f(\x_1,\ldots,\x_d):=\sum_{i_1\in\setint[m_1], \ldots,i_d\in \setint[m_d]}
 f_{i_1,\ldots,i_d}x_{i_1,1}\ldots
 x_{i_d,d}\enspace ,\\
 \x_j=(x_{1,j},\ldots,x_{m_j,j})\trans\in\R^{m_j},\;j\in \setint[d].
 \nonumber
 \end{eqnarray}
Here, and in the sequel, we write $\setint[d]$ for the set $\{1,\ldots,d\}$.
We assume the nontrivial case $d\ge 2, m_j\ge 2, j\in\setint[d]$.
 The above form induces the tensor
 $\cF:=[f_{i_1,\ldots,i_d}]\in \R^{m_1\times\dots\times m_d}$.
 We call $f$ \emph{nonnegative} if the corresponding tensor $\cF$ is
 nonnegative, denoted by $\cF\ge 0$, meaning that all the entries of $\cF$ are
 nonnegative.
 For $\u\in \R^m$ and $p\in (0,\infty]$ denote by $\|
 \u\|_p:=(\sum_{i=1}^m |u_i|^p)^{\frac{1}{p}}$ the $p$-norm of
 $\u$.  Let $\rS_{p,+}^{m-1}:=\{\0\le \u\in\R^m,\;\|\u\|_p=1\}$ be the
 $m-1$ dimensional unit sphere in the $\ell_p$ norm restricted to $\R_+^m$.  Let
 $p_1,\ldots,p_d\in (1,\infty)$ and consider a critical point
 $(\xib_1,\ldots,\xib_d)\in
 \rS_{p_1,+}^{m_1-1}\times\ldots\times\rS_{p_d,+}^{m_d-1}$ of
 $f|\rS_{p_1,+}^{m_1-1}\times \ldots\times\rS_{p_d,+}^{m_d-1}$.
 It is straightforward to show
 that each critical point satisfies the following equality~\cite{Lim05}:
 \begin{eqnarray}\label{critveceq}
\sum\exceptvar{j}
 f_{i_1,\ldots,i_d} x_{i_1,1}\ldots
 x_{i_{j-1},j-1}x_{i_{j+1},j+1}\ldots x_{i_d,d}=\lambda
 x_{i_j,j}^{p_j-1},\\
 i_j\in\setint[m_j],\quad
\x_j\in\rS_{m_j,+}^{p_j-1},\quad j\in \setint[d].\nonumber
 \end{eqnarray}
 Note that for $d=2$ and $p_1=p_2=2$, $\xib_1,\xib_2$ are the
 left and right singular vectors of the nonnegative matrix
 $\cF\in\R^{m_1\times m_2}$.

In the case of nonnegative matrices, irreducibility can be defined
in two equivalent ways, either by requiring the directed graph
associated with the matrix to be strongly connected,
or by requiring that there is no non-trivial \emph{part}
(relative interior of a face) of the standard positive cone
that is invariant
by the action of the matrix, i.e., that the matrix cannot
be put in upper block triangular form by applying the same permutation
to its rows and columns. In the case of tensors, and more generally,
of polynomial maps, both approaches can be extended, leading to distinct notions.

 In the present setting, the tensor $\cF$ is associated with an {\em undirected} $d$-partite graph
 $G(\cF)=(V,E(\cF))$, the vertex set of which is the disjoint union
$V=\cup_{j=1}^d V_j$, with $V_j=\setint[m_j], j\in \setint[d]$.
The edge $(i_k,i_l)\in V_k\times V_l, k\ne l$ belongs
to $E(\cF)$ if and only if $f_{i_1,i_2,\ldots,i_d}>0$ for some $d-2$
 indices $\{i_1,\ldots,i_d\}\backslash\{i_k,i_l\}$.
The tensor $\cF$ is
 called \emph{weakly irreducible} if the graph $G(\cF)$ is connected.
 We call $\cF$ \emph{irreducible} if for each proper nonempty subset
 $\emptyset \ne I\subsetneqq V$,
 the following
 condition holds:  Let $J:=V\backslash I$.  Then there exists
 $k\in [d]$, 
 $i_k\in I\cap V_k$ and $i_j\in J\cap V_j$ for each $j\in\setint[d]
 \backslash\{k\}$ such that $f_{i_1,\ldots,i_d}>0$.
 Our definition of irreducibility agrees with \cite{CPZ08,NQZ09}.
 We will show that if $\cF$ is weakly irreducible then $\cF$ is
 irreducible.

 The main result of this paper gives sufficient conditions on the uniqueness of positive
 solution of the system (\ref{critveceq}), which was studied in \cite{Lim05}, for weakly irreducible and irreducible nonnegative tensors.
 \begin{theo}\label{maintheo}  Let $f:\R^{m_1}\times\ldots\times\R^{m_d}\to\R$
 be a nonnegative multilinear form.  Assume that $\cF$ is weakly
 irreducible and $p_j\ge d$ for $j\in [d]$. 
 Then the system
 (\ref{critveceq}) has a unique positive solution $(\x_1,\ldots,\x_d)>\0$.
 If $\cF$ is irreducible
 then the system
 (\ref{critveceq}) has a unique solution $(\x_1,\ldots,\x_d)$,
 which is necessarily positive.
 \end{theo}
 For $d=2$ and $p_1=p_2=2$ this theorem is a simplified version of the classical
 Perron-Frobenius theorem for the symmetric matrix
 \[
 A=\left[\begin{array}{cc}0&\cF\\\cF\trans&0\end{array}\right]
 \enspace .
 \]
 For $d\ge 3$ and $p_1=\ldots=p_d=d$ our theorem does not follow
 from the main result in~\cite{CPZ08}, in which the tensor is required
to be irreducible (rather than weakly irreducible).
We also give examples where the statement of the theorem fails
 if the conditions $p_j \ge d$ for $j\in\setint[d]$
are not  satisfied.

Theorem~\ref{maintheo} deals with the eigenproblem~\eqref{critveceq},  which is of a variational nature. In other words, the polynomials at the left-hand side of~\eqref{critveceq} determine the gradient of a single form. However,
our next result, Theorem~\ref{mtheoa} below, shows that similar conclusions holds for more general polynomial eigenvalue problems. We also address computational aspects: in Corollary~\ref{cor-nussbaum} and~\ref{coro-spectralgap}, we give a sufficient condition, %
of a combinatorial nature (weak primitivity), which guarantees
that the power algorithm converges to a normalized eigenvector, and
we derive a spectral gap type formula for the asymptotic convergence rate.
We note that a similar power algorithm for nonnegative tensors was
introduced in~\cite{NQZ09} (in a slightly
more special context), where extensive numerical tests were given.
The sequence produced by the present power
algorithm can be seen to coincide with the one of~\cite{NQZ09}, up to a
normalization factor, and therefore, the present results imply that
the algorithm of~\cite{NQZ09} does converge,
under the weak primitivity condition.

We now describe briefly the organization of this paper.  In \S2 we recall the results of \cite{GG04,Nus88}
on existence and uniqueness of positive eigenvectors of homogeneous monotone maps that act on the interior of the cone $\R_+^n$.  In \S3 we prove Theorem~\ref{maintheo} by constructing
a homogeneous monotone map of degree one and invoking the
results of \cite{GG04,Nus88}.  The extension to eigenproblems
involving polynomials maps with nonnegative coefficients
is presented in~\S\ref{sec-ext}. We derive in particular
from a result of~\cite{nuss86} an analogue of Collatz-Wielandt's
minimax characterization of the Perron eigenvalue.
The power algorithm is analysed in \S\ref{sec-algo}.
Finally, in \S\ref{sec-ex} we give numerical
examples showing that the conclusion of Theorem~\ref{maintheo} no longer holds for $p_1=\ldots=p_d<d$.

 \section{Eigenvectors of homogeneous monotone maps on $\R_+^n$}

 Let $\R_+:=[0,\infty),\R_{>0}:=(0,\infty)$,
 $\R_+^n=(\R_+)^n, \R_{>0}^n=(\R_{>0})^n$ be the cone of nonnegative vectors
 and its interior respectively. For $\x,\y\in\R^n$ we denote $\y\ge \x, \y\gneq \x,\y>\x$ if $\y-\x\in\R_+^n,
 \y-\x\in\R_+^n\setminus\{\0\}, \y-\x\in\R_{>0}^n$ respectively.
 Recall that the Hilbert metric on $\R_{>0}^n$ is the map $\dist:\R_{>0}^n\times\R_{>0}^n\to [0,\infty]$ defined as follows. Let $\x=(x_1,\ldots,x_n)\trans,\y=(y_1,\ldots,y_n)\trans >\0$.  Then
 $\dist(\x,\y)=\max_{i\in[n]}\log\frac{y_i}{x_i}-\min_{i\in[n]}\log\frac{y_i}{x_i}$.  Note that $\dist(\x,\y)=\dist( \y,\x)\ge 0$ and equality holds if and only if $\x$ and $\y$ are colinear.
 Furthermore the triangle inequality holds.  The Hilbert metric has a simple interpretation.
 Let $\alpha,\beta>0$ and assume that $\0<\alpha\x\le \y\le \beta\x$.  Then $\dist(\x,\y)\le \log\frac{\beta}{\alpha}$.  Equality is achieved when $\alpha=\alpha_{\max}$ and $\beta=\beta_{\min}$ are the maximal and the minimal possible satisfying the above inequality.
 Fix $\psi\in\R_+^n\setminus\{\0\}$ and consider the open polytope
 $\Sigma(\psi)=\{\x>\0, \psi\trans \x=1\}$.  Then the Hilbert metric is indeed a metric on $\Sigma(\psi)$.
 Equivalently, the Hilbert metric is a metric on the space of rays $\P\R_{>0}^n$ in $\R_{>0}^n$. Note that both spaces are complete for Hilbert's metric.

 In this section we assume that $\bF=(F_1,\ldots,F_n)\trans: \R_{>0}^n\to\R_{>0}^n$ satisfies the following properties.  First, $\bF$ is homogeneous, (or homogeneous of degree one), i.e. $\bF(t\x)=t\bF(\x)$ for each $t>0$ and $\x> \0$.  Second, $\bF$ is monotone, i.e. $\bF(\y)\ge \bF(\x)$ if $\y\ge\x> \0$.  So $\bF$ can be viewed as a map $\hat \bF:
 \P\R_{>0}^n\to\P\R_{>0}^n$.
 $\x> \0$ is called an \emph{eigenvector} of $\bF$ in $\R_{>0}^n$ if $\bF(\x)=\lambda\x$ for some $\x>\0$ and an \emph{eigenvalue} $\lambda > 0$.

 It is well known that $\bF$ is nonexpansive with respect to Hilbert metric, i.e. $\dist(\bF(\x),\bF(\y))\le
 \dist(\x,\y)$.  Indeed
 \begin{eqnarray*}
 \alpha_{max}\x\le \y\le \beta_{\min}\x\Rightarrow \alpha_{\max}\bF(\x)=\bF(\alpha_{\max}\x)\le
 \bF(\y)\le \bF(\beta_{\min}\x)=\beta_{\min}\bF(\x)\Rightarrow\\
 \dist(\bF(\x),\bF(\y))\le \log\frac{\beta_{\min}}{\alpha_{\max}}=\dist(\x,\y).
 \end{eqnarray*}

 Assume that $\bF$ is a contraction, i.e. $\dist(\bF(\x),\bF(\y))\le
 K\dist(\x,\y)$ for some $K\in [0,1)$.  Then the Banach fixed point theorem implies that $\hat \bF$
 has a unique fixed point in $\P\R_{>0}^n$.  Hence $\bF$ has a unique eigenvector
 in $\R_{>0}^n$.

 Another general criteria for existence of eigenvectors $\x\gneq \0$ can be obtained
 by using the Brouwer fixed point theorem.
 Assume that $\bF$ extends to a continuous map $\bF(\x):\R_+^n\to \R_+^n$ and suppose that
 $\bF^{-1}(\0)=\0$.  Let $\psi>\0$.  Define
 $\bG:\textrm{Closure }\Sigma(\psi)\to\textrm{Closure }\Sigma(\psi)$:
 \begin{equation}\label{defG}
 \bG(\x)=\frac{1}{\psi\trans \bF(\x)} \bF(\x),
 \end{equation}
 for $\x\gneq \0$.
 Then $\bG$ has a fixed point $\y\in\textrm{Closure }\Sigma(\psi)$ which is an eigenvector of $\bF$ in $\R_+^n$.
 However $\y$ may be on the boundary of $\Sigma(\psi)$.

 Theorem 2 in \cite{GG04} gives a sufficient condition for existence of an eigenvector of $\bF$
 in $\R_{>0}^n$.
 For $u\in (0,\infty)$ and $J\subseteq [n]$ denote by $\u_J=(u_1,\ldots,u_n)\trans>\0$ the following vector.
 $u_i=u$ if $i\in J$ and $u_i=1$ if $i\not\in J$.  Then $F_i(\u_J)$ is a nondecreasing function in $u$.
 The associated di-graph $\cG(\bF)\subset [n]\times [n]$ is defined as follows.
 $(i,j)\in\cG(\bF)$ if and only if $\lim_{u\to\infty} F_i(\u_{\{j\}})=\infty$.
 \begin{theo}\label{thm2gg04} (\cite[Theorem 2]{GG04}) Let $\bF: \R_{>0}^n\to\R_{>0}^n$ be homogeneous and monotone. If $\cG(\bF)$ is strongly connected then $\bF$ has an eigenvector in $\R_{>0}^n$.
 \end{theo}

 Theorem 2.5 of \cite{Nus88} gives a simple sufficient condition on the uniqueness of an eigenvector $\u>\0$
 of a homogeneous monotone $\bF$.  Assume that $\bF$ is $C^1$ in an open neighborhood $O\subset \R_{>0}^n$ of $\u$.  I.e. $\bF$ is continuous and has continuous first partial derivatives in $O$.  Consider the matrix
 $\rD\bF(\u)=[\frac{\partial F_i}{\partial x_j}(\u)]_{i=j=1}^n$.  Since $\bF$ is monotone it follows that
 $\rD\bF(\u)$ is a nonnegative square matrix.
 \begin{theo}\label{nusth2.5}(\cite[Theorem 2.5]{Nus88})  Let $\bF:\R_{>0}^n\to\R_{>0}^n$ be homogeneous and monotone.  Assume that
 $\u>\0$ is an eigenvector of $\bF$.  Suppose that $\bF$ is $C^1$ in some open neighborhood of $\u$.
 Assume that $A=\rD\bF(\u)$ is either nilpotent or has a positive spectral radius $\rho$ which is a simple root of the characteristic polynomial of $A$.  Then $\u$ is a unique eigenvector of $\bF$ in $\R_{>0}^n$.
 \end{theo}
 Corollary 2.5 \cite{Nus88} deals with the case where $A$ is primitive.
 \begin{theo}\label{nuscor2.5} (\cite[Corollary 2.5]{Nus88})   Let $\bF:\R_{>0}^n\to\R_{>0}^n$ be homogeneous and monotone.  Assume that $\u>\0$ is an eigenvector of $\bF$.
 Suppose that $\bF$ is $C^1$ in some open neighborhood of $\u$.
 Assume that $A=\rD\bF(\u)$ is primitive.  Let $\psi\gneq \0$ and assume that $\psi\trans\u=1$.
 Then $\u>\0$ is the unique eigenvector of $\bF$ in $\R_{>0}^n$ satisfying the condition $\psi\trans\u=1$.
 Define $\bG:\R_{>0}^n\to \R_{>0}^n$ as in (\ref{defG}).
 Then $\lim_{m\to\infty} \bG^{\circ m}(\x)=\u$ for each $\x\in\R_{>0}^n$.
 \end{theo}

 \section{Proof of the main theorem}
 \begin{lemma}\label{indecirred}
 Let $f$ be a $d$-multilinear nonnegative form for $d\ge 2$.
 Assume that $\cF$ is irreducible.  Then $\cF$ is
 weakly irreducible.  For $d=2$, $\cF$ is irreducible if and only if
 $\cF$ is weakly irreducible.
 \end{lemma}
 \proof
 Assume to the contrary that $\cF$ is not weakly irreducible.
 So the graph $G(\cF)$ is not connected.
 So there exists $\emptyset \ne I\subsetneqq V$
 such that there is no edge from $I$ to $J=V\backslash I$.
 Let $I_k:=I\cap V_k$ for $k\in\setint[d]$.
Let $I_k'$ be
 defined as follows. $I_k'=I_k$ if $I_k\ne V_k$ and
 $I_k'=\setint[m_k-1]$
 if $I_k=V_k$.
 Note that $I_l'=I_l$ for some $l$.  Let $I'=\cup_{k=1}^d I_k',
 J'=V\backslash I'$.  Since $\cF$ is irreducible there
 exists $i_k\in I'\cap V_k$ and $i_j\in J'\cap V_j$ for $j\ne
 k$ such that $f_{i_1,\ldots,i_d}>0$.  Since $I_l'=I_l$ it
 follows that $i_l\in J$.  So $(i_k,i_l)\in E(\cF)$,
 which contradicts our assumption.

 Assume that $d=2$.  It is straightforward to show that
 if $\cF$ weakly irreducible then $\cF$ is irreducible.
 \qed

 We identify $\rC^{o}:=\R_{>0}^{m_1}\times
 \dots \times \R_{>0}^{m_d}$, the interior of the cone $\rC=\R_+^{m_1+\ldots+m_d}$,
with $\R_{>0}^{m_1+\cdots+m_d}$.
 Let
 \begin{equation}\label{definp}
 p:=\max(p_1,\ldots,p_d).
 \end{equation}
 For a nonnegative $d$-form (\ref{defdmulfrm})
 define the following homogeneous map $\bF:\rC\setminus\{\0\}\to\rC\setminus\{\0\}$ of degree one:
 \begin{eqnarray*}\nonumber
\lefteqn{ F((\x_1,\ldots,\x_d))_{i_j,j}= }\\
 &&\!\!\!\!\!\!\!\Bigg (x_{i_j,j}^{p-p_j}\|\x_j\|_{p_j}^{p_j-d}
 \sum\exceptvar{j}
 f_{i_1,\ldots,i_d} x_{i_1,1}\ldots
 x_{i_{j-1},j-1}x_{i_{j+1},j+1}\ldots x_{i_d,d}\Bigg )^{\frac{1}{p-1}},
\label{defFmap}
\\
&& \qquad \qquad\qquad\qquad\qquad\qquad\qquad\qquad\qquad\qquad\qquad\qquad\qquad i_j\in\setint[m_j]\quad
j\in\setint[d].\nonumber
 \end{eqnarray*}
 To avoid trivial cases we assume that $F_{i_j,j}$ is not
 identically zero for each $j$ and $i_j$, i.e.
 \begin{equation}\label{nonvancon}
 \sum\exceptvar{j}
f_{i_1,\ldots,i_d} > 0,  \textrm{ for all } i_j\in\setint[m_j]
\;j\in\setint[d].
 \end{equation}
 Note that if $\cF$ is weakly irreducible, then this condition is
 satisfied.
 Let $\y=(\x_1,\ldots,\x_d)$.
 In what follows we assume the condition
 \begin{equation}\label{pjcond}
 p_j\ge d, \quad j\in\setint[d].
 \end{equation}
 Then $\bF$ is monotone on $\rC^o$.
 Recall the definition of the di-graph
 $$\cG(\bF)=(V,E(\bF)),\quad E(\bF)\subset
 (V_1\times\ldots\times V_d)^2$$
 associated with $\bF$, as in \S2.  Let $\1_j=(1,\ldots,1)\trans,
 \e_{k,j}=(\delta_{k1},\ldots,\delta_{km_j})\trans\in \R^{m_j}$
 for $j\in\setint[d]$.
 Then for $i_k\in V_k, i_l\in V_l$:
 $$(i_k,i_l)\in E(\bF) \iff \lim_{t\to\infty}
 F_{i_k,k}((\1_1,\ldots,\1_d)+t(\0,\ldots,\0,\e_{i_l,l},\0,\ldots,\0))=\infty.
 $$
 Equivalently, there is a di-edge from $i_k\in V_k$ to $i_l\in V_l$ if and only if
 the variable
 $x_{i_l,l}$ effectively appears in the expression of $F_{i_k,k}$.
 The following lemma is deduced straightforwardly.
 \begin{lemma}\label{charFgraph}
 Let $f$ be a nonnegative multilinear form given by
 (\ref{defdmulfrm}).  Assume that the conditions
 (\ref{nonvancon}) and (\ref{pjcond}) hold.  Then, $(r,s)$
is a di-edge of the di-graph  $\cG(\bF)=(V,E(\bF))$ if and only if
at least one of the following conditions holds:
 \begin{enumerate}
 \item\label{charFgraph1}
$r=i_k\in V_k, s=i_l\in V_l$, with $k\neq l$, and we can find
$d-2$ indices $i_j\in V_j$, for $1\leq j\leq d$, $j\not\in
\{k,l\}$, such that
 $f_{i_1,i_2,\ldots,i_d}>0$;
 \item\label{charFgraph2}
$r,s$ belong to $V_k$ and $p_k>d$;
 \item\label{charFgraph3}
$r=s$ belongs to $V_k$ and $p>p_k$.
 \end{enumerate}
 In particular, if $\cF$ is weakly irreducible then $\cG(\bF)$ is
 strongly connected.

 \end{lemma}
 \begin{theo}\label{theorem1}  Let $f$ be a nonnegative multilinear form given by
 (\ref{defdmulfrm}).  Assume that the conditions
 (\ref{nonvancon}) and (\ref{pjcond}) hold.  Let $p=\max(p_1,\ldots,p_d)$.  Suppose
 furthermore that $\cG(\bF)$ is strongly connected.  Then $\bF$ has
 a unique positive eigenvector
 up to a positive multiple.
 I.e.\ there exist positive vectors $\0<\xib_j\in
 \R^{m_j} ,j\in\setint[d]$
and a positive eigenvalue $\mu$
 with the following properties.
 \begin{enumerate}
 \item\label{theorem1a}  $\bF((\xib_1,\ldots,\xib_d))=\mu
 (\xib_1,\ldots,\xib_d)$.  In particular,
 \begin{equation}\label{theorem1c}
 f(\xib_1,\ldots,\xib_d)=\mu^{p-1}\|\xib_j\|_{p_j}^{d} \textrm{ for }
 j\in\setint[d].
 \end{equation}
 \item\label{theorem1b}  Assume that
 $\bF((\x_1,\ldots,\x_d))=\alpha (\x_1,\ldots,\x_d)$ for some
 $(\x_1,\ldots,\x_d)>\0$.  Then $\alpha=\mu$ and
 $(\x_1,\ldots,\x_d)=t(\xib_1,\ldots,\xib_d)$ for some $t>0$.

 \end{enumerate}
 Let
 $p_1=\ldots=p_d=p\ge d$ and
 assume furthermore that $\cF$ is irreducible.  Then
 the conditions \ref{theorem1a}-\ref{theorem1b} hold.
 Suppose that $\bF((\x_1,\ldots,\x_d))=\alpha (\x_1,\ldots,\x_d)$ for some
 $(\x_1,\ldots,\x_d)\gvertneqq\0$ and $\|\x_j\|_p>0$ for
 $j\in\setint[d]$.
 Then $\alpha=\mu$ and
 $(\x_1,\ldots,\x_d)=t(\xib_1,\ldots,\xib_d)$ for some $t>0$.
 \end{theo}
 \proof  The existence of a positive eigenvector
 $(\xib_1,\ldots,\xib_d)$ of $\bF$ follows from Theorem \ref{thm2gg04}.
 We next derive the uniqueness of $(\xib_1,\ldots,\xib_d)$
 from Theorem \ref{nusth2.5}.  Clearly each coordinate of $\bF$ is a smooth function on $\rC^o$.
 Let $\bF(\y)=t\y, \y\in\rC^o$.  Let
 $A=\rD \bF(\y)\in\R_+^{(m_1+\cdots+ m_d)\times (m_1+\cdots +m_d)}$.
 The di-graph $G(A)$,
 induced by the nonnegative entries of $A$, is equal to
 $\cG(\bF)$.  The assumption that $\cG(\bF)$ is strongly connected
 yields that $\rD \bF(\y)$ is irreducible.  Hence $\bF$ has a unique
 positive eigenvector $\y=(\x_1,\ldots,\x_d)$,
 up to a product by a positive scalar.

 It is left to show the condition (\ref{theorem1c}) for the
 eigenvector $\y$.
 Raise the equality $F_{i_j,j}(\x_1,\ldots,\x_d)=\mu x_{i_j,j}$
 to the power $p-1$ and divide by $x_{i_j,j}^{p-p_j}$ to obtain
 the equality
 $$\|\x_j\|_{p_j}^{p_j-d}
 \sum\exceptvar{j}
 f_{i_1,\ldots,i_d} x_{i_1,1}\ldots
 x_{i_{j-1},j-1}x_{i_{j+1},j+1}\ldots
 x_{i_d,d}=\mu^{p-1}x_{i_j,j}^{p_j-1}.$$
 Multiply this equality by $x_{i_j,j}$ and sum on
 $i_j=1,\ldots,m_j$ to obtain
 $$\|\x_j\|_{p_j}^{p_j-d}f(\x_1,\ldots,\x_d)=\mu^{p-1}\|\x_j\|_{p_j}^{p_j},
 \;j\in\setint[d],$$
 which is equivalent to (\ref{theorem1c}).
 Assume that $p_1=\ldots=p_d=p$ and $\cF$ is irreducible.  Lemma \ref{indecirred}
 yields that $\cF$ is weakly irreducible.
 Lemma \ref{charFgraph} implies that $\cG(\bF)$ is strongly
 connected. Hence the conditions
 \ref{theorem1a}-\ref{theorem1b} hold.

 Assume that $p_1=\ldots=p_d=p\ge d$, $\bF$ is irreducible and $\bF((\z_1,\ldots,\z_d))=
 \alpha (\z_1,\ldots,\z_d)$ for some
 $(\z_1,\ldots,\z_d)\gneq\0$, where $\|\z_j\|_p>0$ for
 $j\in\setint[d]$.
 Suppose furthermore that $(\z_1,\ldots,\z_d)$
 is not a positive vector.
 Let $\emptyset \ne I,J\subset \cup_{j=1}^d V_i$ be the set of indices where
 $(\z_1,\ldots,\z_d)$ have zero and positive coordinates respectively.
 I.e. $i_k\in I\cap V_k$ if and only if $z_{i_k,k}=0$.
 Since $\|\z_k\|>0$ for each $k\in\setint[d]$
 it follows that
 $I\cap V_k\ne V_k$, and
 for each $i_k\in I\cap V_k$ we have the equality
 $\frac{(F_{i_k}(\z_1,\ldots,\z_d))^{p-1}}{\|\z_j\|_p^{p-d}}=0$.
 Hence $f_{i_1,\ldots,i_d}=0$ for each $i_k\in I\cap V_k$ and $i_j\in J\cap V_j$
 for each $j\in\setint[d]\backslash \{k\}$.
 This contradicts the assumption that $\cF$ is irreducible.
 \qed

 \textbf{Proof of Theorem \ref{maintheo}}  Apply Theorem
 \ref{theorem1}.  Normalize the positive eigenvector
 $(\xib_1,\ldots,\xib_d)$ by the condition $\|\xib_1\|_{p_1}=1$.
 Then the first condition of (\ref{theorem1c}) yields that
 $f(\xib_1,\ldots,\xib_d)=\mu^{p-1}$.  The condition
 (\ref{theorem1c}) for $j>1$ yields that $\|\xib_j\|_{p_j}=1$.
 Hence the equality $\bF((\xib_1,\ldots,\xib_d))=\mu
 (\xib_1,\ldots,\xib_d)$ implies (\ref{critveceq}) with
 $\lambda=\mu^{p-1}$.

 Assume now that (\ref{critveceq}) holds.
 Then
 $\bF((\x_1,\ldots,\x_p))=\lambda^{\frac{1}{p-1}}(\x_1,\ldots,\x_p)$.
 Hence, if $(\x_1,\ldots,\x_p)>\0$ it follows that
 $(\x_1,\ldots,\x_p)=(\xib_1,\ldots,\xib_d)$.

 Assume now $\cF$ is irreducible.
 We claim that
 (\ref{critveceq}) implies that $\lambda>0$ and $\x_j>\0$ for
 $j\in\setint[d]$.
 Assume to the contrary that $(\x_1,\ldots,\x_d)$
 is not a positive vector.
 Let $\emptyset \ne I,J\subset \cup_{j=1}^d V_i$ be the set of indices where
 $(\x_1,\ldots,\x_d)$ have zero and positive coordinates respectively.
 I.e. $i_k\in I\cap V_k$ if and only if $z_{i_k,k}=0$.
 Since $\|\x_k\|_{p_k}>0$ for each $k\in\setint[d]$
it follows that
 $I\cap V_k\ne V_k$, and
 for each $i_k\in I\cap V_k$ we have the equality
 $$\sum\exceptvar{k}
 f_{i_1,\ldots,i_d} x_{i_1,1}\ldots
 x_{i_{k-1},k-1}x_{i_{k+1},k+1}\ldots x_{i_d,d}=0.$$
 Hence $f_{i_1,\ldots,i_d}=0$ for each $i_k\in I\cap V_k$ and $i_j\in J\cap V_j$
 for each $j\in\setint[d]\backslash \{k\}$.
 This contradicts the assumption that $\cF$ is irreducible.
 So $(\x_1,\ldots,\x_d)$ must be
 a positive vector, and so $\lambda>0$.  The previous
 arguments show that the system
 (\ref{critveceq}) has a unique solution $(\x_1,\ldots,\x_d)$,
 which is positive.
 \qed

 \section{Extension: Perron-Frobenius theorem for nonnegative polynomial maps}\label{sec-ext}
 Let $\bP=(P_1,\ldots,P_n)\trans: \R^n\to \R^n$ be a polynomial map.  We assume that each $P_i$ is a 
polynomial of degree $d_i\ge 1$,
 and that the coefficient of each monomial in $P_i$ is nonnegative.  So $\bP:\R_+^n\to \R_+^n$.  We associate with $\bP$ the following di-graph $G(\bP)=(V,E(\bP))$, where $V=[n] $ and $(i,j)\in E(\bP)$ if
the variable $x_j$ effectively appears in the expression of $P_i$,
or if this expression contains a monomial with degree $<d_i$ (note that the latter case may only occur when $P_i$ is not homogeneous).
 We call $\bP$ \emph{weakly irreducible} if $G(\bP)$ is strongly connected.
To each subset $I\subset [n]$, is associated a \emph{part} $Q_I$,
which consists of the vectors $\x=(x_1,\ldots,x_n)\trans\in \R_+^n$ such that $x_i>0$
iff $i\in I$.
We say that the polynomial map $\bP$ is \emph{irreducible} if there is no part of $\R_+^n$
that is invariant by $\bP$, except the trivial parts $Q_{\emptyset}$
and $Q_{[n]}$. Observe that when $\bP$ is the polynomial
map appearing at the left-hand side of~\eqref{critveceq}, this definition
of irreducibility coincides
with the one made in Section~\ref{intro}. Observe also that if $\bP$ is irreducible, then, it is weakly irreducible (the proof is a simplified version
of the proof of Lemma~\ref{indecirred}).
We have the following extension of Theorem~\ref{maintheo}.
 \begin{theo}\label{mtheoa}  Let $\bP=(P_1,\ldots,P_n)\trans: \R^n\to \R^n$ be a polynomial map, where each $P_i$ is a 
 polynomial of degree $d_i\ge 1$
 with nonnegative coefficients.  Let $\delta_1,\ldots,\delta_n\in(0,\infty)$ be given and assume that $\delta_i\ge d_i, i\in\setint[n]$.
 Consider the system
 \begin{equation}\label{Psystem}
 P_i(\x)=\lambda x_i^{\delta_i}, \quad i\in\setint[n], \quad \x\ge \0.
 \end{equation}
Assume that $\bP$ is weakly irreducible.
Then for each $a,p>0$ there exists a unique positive vector $\x>\0$, depending on $a,p$, satisfying (\ref{Psystem}) and the condition $\|\x\|_p=a$.
 Suppose furthermore that
$\bP$ is irreducible.  Then the system (\ref{Psystem}) has a unique solution, depending on $a,p$  satisfying $\|\x\|_p=a$, and all the coordinates of this solution are positive.
 \end{theo}
 \proof  We can write
every coordinate of $P$ as a sum
\[
P_i(\x)=\sum_{\j\in \Z_+^n}a_{i\j} \x^\j \enspace,
\]
where $\Z_+$ is the set of nonnegative integers, $\j=(j_1,\ldots,j_n)$ is a multi-index, $\x^\j:=x_1^{j_1}\cdots x_n^{j_n}$,
$a_{i\j}\geq 0$, and $a_{i\j}=0$ except for a finite number of values of $\j$.
We set $|\j|:=j_1+\dots+j_n$.

Let $\delta=\max(\delta_1,\ldots,\delta_n)$,
and consider the following homogeneous monotone map $\bF=(F_1,\ldots,F_n)\trans:\R_+^n \to \R^n_+$ given by
 \begin{equation}\label{defFS5}
 F_i(\x)=\Big(\sum_{\j\in\Z_+^n}
a_{i\j}x_i^{\delta-\delta_i}\big(\frac{\|\x\|_p}{a}\big)^{\delta_i-|\j|}\x^\j
\Big)^{\frac{1}{\delta}}, \quad i\in\setint[n].
 \end{equation}
Observe that the di-graph of $\bF$, in the sense of Theorem~\ref{thm2gg04},
coincides with the di-graph of $\bP$, except perhaps for loops
(di-edges $i\to i$). Indeed, for every variable $x_k$ effectively appearing
in the expansion of $\bP_i$, there is a di-edge from $i$ to $k$ in the di-graph of $\bP$. Moreover, if there is one monomial $a_\j\x^\j$ in $P_i$ of degree $|\j|<\delta_i$, the presence of the term $\|\x\|_p$ in the construction of $\bF$ yields a di-edge between $i$ and every $k\in[n]$. Finally, if $\delta>\delta_i$, there is a di-edge from $i$ to $i$ in the same di-graph.

Thus, if $\bP$ is weakly irreducible, the di-graph
of $\bF$ is strongly connected, and so, there exists a vector $\x>0$
and a scalar $\mu>0$ such that $F(\x)=\mu \x$. Since $F$ is positively homogeneous, we may normalize $\x$ so that $\|\x\|_p=a$. Then, we readily deduce
from~(\ref{defFS5}) that
\[
\sum_{\j\in\Z_+^n}a_{i\j}x_i^{\delta-\delta_i}\x^\j=\mu^\delta x_i^\delta \enspace.
\]
 Since we assumed that $\x>0$, after multiplying by $x_i^{\delta_i-\delta}$ both
sides, we arrive at
$P_i(\x)=\mu^\delta x_i^{\delta_i}$,
with $\lambda:=\mu^\delta$,
showing that $\x$ satisfies~(\ref{Psystem}).
Conversely, any solution $\x$ of~\eqref{Psystem}
such that $\|\x\|_p=a$ is also an eigenvector of $\bF$. Consider any such solution
$\x>0$, and let $A:=D\bF(\x)$. The di-graph of the matrix $A$ coincides
with the di-graph of $\bF$, which, as noted above, contains the di-graph of $\bP$. Then, it
follows from Theorem~2.2 that $\bF$ has a unique positive eigenvector,
up to a multiplicative constant, and so there is a unique solution $\x>0$
of~\eqref{Psystem}.

Assume now that $\x\gneq 0$ is a solution of~\eqref{Psystem}
such that $\|\x\|_p=a$,
and let $I:=\{i\in[n]\mid x_i\neq 0\}$. Then, it follows from~\eqref{Psystem}
that the part $Q_I$ is invariant by $\bP$. Hence, if $\bP$ is irreducible,
$I=[n]$. Then, the uniqueness of $\x$ follows from the first part of the proof.
%
\qed

We now turn our attention to the case where all the entries of $P$ are homogeneous. Then, we will derive from a general result of~\cite{nuss86} a minimax characterization
of the eigenvalue of $P$ in~\eqref{Psystem}, similar to the classical Collatz-Wielandt formula in
Perron-Frobenius theory.

If $\bF$ is homogeneous monotone map of $\R_+^n$ to itself,
the \emph{cone spectral radius} of $\bF$, denoted by
$\rho(\bF)$, is
the greatest scalar $\mu$ such that there exists a nonzero
vector $\u\in \R_+^n$ such that $\bF(\u)=\mu \u$. We shall refer
to $\mu$ and $\u$ as a nonlinear \emph{eigenvalue} and \emph{eigenvector} of $\bF$, respectively.
We shall also use the following generalizations of
the Collatz-Wielandt functions arising classically in Perron-Frobenius
theory:
\begin{eqnarray*}
\cw(\bF)
&=& \inf\{\mu \mid \exists \u\in \interior\R_+^n,\; \bF(\u)\leq \mu \u\} \enspace ,\\
\cwl(\bF)&=& \sup\{\mu \mid \exists \u\in \R_+^n\setminus\{0\},\; \bF(\u)\geq \mu \u\}  \enspace .
\end{eqnarray*}
Nussbaum proved in~\cite[Theorem~3.1]{nuss86} that
\begin{equation}
\rho(\bF)=\cw(\bF)\label{e-nussbaum}
\enspace.
\end{equation}
>From this,
one can deduce that
\begin{equation}
\cwl(\bF)=\cw(\bF) \enspace ,\label{coro-nuss}
\end{equation}
see~\cite[Lemma~2.8]{AGG09}.
Then, we obtain the following Collatz-Wielandt type
property for nonnegative polynomial maps.
\begin{corol}\label{coro-cw}
Let $\bP=(P_1,\ldots,P_n)\trans: \R^n\to \R^n$ be a polynomial map, where each $P_i$ a homogeneous polynomial of degree $d\ge 1$
 with nonnegative coefficients.
Assume that $\bP$ is weakly irreducible. Then, the unique scalar $\lambda$
such that there is a positive vector $\u$ with $P_i(\u)=\lambda u_i^{d}$
for all $i\in\setint[n]$ satisfies:
\begin{eqnarray}
\lambda = \inf_{\x\in \interior\R_+^n} \max_{i\in [n]}
\frac{P_i(\x)}{x_i^d}
=\sup_{\x\in\R_+^n\setminus\{0\}} \min_{\scriptstyle i\in [n]\atop \scriptstyle x_i \neq 0}
\frac{P_i(\x)}{x_i^d}\enspace .
\label{e-cw-functionprime}
\end{eqnarray}
\end{corol}
\begin{proof}
Define the map $\bF$ as in the proof of Theorem~4.1, so that
\begin{equation}
F_i(\x)= (P_i(\x))^{1/d} \enspace .
\label{eq-1overd}
\end{equation}
By Theorem~\ref{mtheoa}, $\bF$ has an eigenvector in the interior of
$\R_+^n$. Let $\mu$ be the associated eigenvalue. By definition
of $\rho(\bF)$ and $\cw(\bF)$, we have $\rho(\bF)\geq \mu\geq \cw(\bF)$.
From~\eqref{e-nussbaum}, we deduce that $\mu=\cw(\bF)$. Observe that
\begin{eqnarray}
\cw(\bF)= \inf_{\x\in\interior\R_+^n}\max_{i\in\setint[n]}\frac{F_i(\x)}{x_i}
=  \inf_{\x\in\interior\R_+^n}\max_{i\in\setint[n]}\frac{P_i(\x)^{\frac{1}{d}}}{x_i}
\enspace .
\label{e-cwforp}
\end{eqnarray}
It follows that $\lambda=\mu^d$ is given by the first expression in~\eqref{e-cw-functionprime}.
A similar reasoning,
this time with the lower Collatz-Wielandt type number $\cwl(\bF)$,
leads to the second expression in~\eqref{e-cw-functionprime}.
\end{proof}
As an immediate consequence, we get the following analogue
of the characterization of the Perron root of an irreducible
nonnegative matrix as the spectral radius.
\begin{corol}
Let $\bP$, $d$ and $\lambda$ be as in Corollary~\ref{coro-cw}.
If $\nu\in\C$ and $\v=(v_1,\ldots,v_n)\trans\in \C^n\setminus\{0\}$ are such that $P_i(\v)=\nu v_i^d$,
for all $i\in\setint[n]$, then $|\nu|\leq \lambda$.
\end{corol}
\begin{proof}
Let $u_i:=|v_i|$ and $\u=(u_1,\ldots,u_n)\trans$. Then, $P_i(\u)\geq |\nu|u_i^d$, and so
\[
 \min_{\scriptstyle i\in [n]\atop \scriptstyle u_i \neq 0}
\frac{P_i(\u)}{u_i^d} \geq |\nu| \enspace .
\]
It follows from~\eqref{e-cw-functionprime} that $\lambda\geq |\nu|$.\qed
\end{proof}
\section{Algorithmic aspects}\label{sec-algo}
The following simple power type algorithm will allow us to compute the vector
$\x$ in~\eqref{Psystem} in an important special case.
Assume as in Corollary~\ref{coro-cw} that each polynomial $P_i$ is
homogeneous of degree $d$, with nonnegative coefficients, and
define the map $\bF$ by~\eqref{eq-1overd}.
Let $\psi>\0$ and consider a sequence $\x^{(k)}=(x_1^{(k)},\ldots,x^{(k)}_n)\trans$
inductively defined by
\[ x^{(k+1)}_i= (\psi\trans\bF(\x^{(k)}))^{-1}F_i(\x^{(k)}), \qquad k=0,1,2,\ldots\]
where $\x^{(0)}$ is an arbitrary vector in the interior of the cone.
We shall say that $\bP$ is \emph{weakly primitive} if the di-graph $G(\bP)$ is strongly connected and if the gcd of the lengths of its circuits is equal to one.

We note that weak primitivity does not imply in general that for each $\x^{(0)}\in\R^n\setminus\{\0\}$
there exists $k\in\N$ such that $\x^{(k)}$ has positive coordinates.  Indeed choose $\bP(\x)=(x_1x_2, x_2^2)\trans: \R_+^2\to\R_+^2$.

The following is readily deduced from a general result of~\cite{Nus88}, see Theorem \ref{nuscor2.5}.
\begin{corol}\label{cor-nussbaum}
Let $\bP$ and $d$ be as in Corollary~\ref{coro-cw}, and assume in addition
that $\bP$ is weakly primitive.  Then, the sequence $\x^{(k)}$ produced by the power
algorithm converges to the unique vector $\u\in \interior \R_+^n$
satisfying $P_i(\u)=\lambda u_i^d$, for $i\in\setint[n]$, and $\psi\trans \u=1$.
\end{corol}
\begin{proof}
The map $\bF$ defined in (\ref{defFS5}) is differentiable, and its derivative
at any point of the interior of $\R_+^n$ is a nonnegative matrix
the di-graph of which is precisely $G(\bP)$. Hence, this nonnegative matrix
is primitive and the assumptions of Theorem \ref{nuscor2.5}
are satisfied. It follows that $\x^{(k)}$ converges to the only
eigenvector $\u$ of $F$ in the interior of $\R_+^n$ such that $\psi\trans \u=1$.
\qed
\end{proof}

The result of~\cite{Nus88} implies that the convergence of the power algorithm is geometric, and it yields a bound on the geometric convergence rate which tends
to $1$ as the distance between $\x^{(0)}$ and $\u$ in Hilbert's projective metric tends to infinity.
We next estimate the asymptotic speed of convergence.
To do so, it is enough to linearize $\bF$ around $\u$, see for example the arguments in \cite{Fr06}. We give a short proof for reader's convenience,
leading to an explicit formula for the rate.
\begin{corol}\label{coro-spectralgap}
Let $\bP$, $d$, $u$ and $\lambda$ be as in Corollary~\ref{cor-nussbaum}, $\bF$ as in (\ref{defFS5}),
let $M:=\bF'(\u)$, and let $r$ denote the maximal modulus of the eigenvalues
of $M$ distinct from $\lambda$. Then, the sequence
$\x^{(k)}$ produced by the power algorithm satisfies
\[
\limsup_{k\to\infty}\|\x^{(k)}-\u\|^{1/k}\leq \lambda^{-1}r \enspace .
\]
\end{corol}
\begin{proof}
For all $\alpha>1$, we have $\alpha\lambda \u = \bF(\alpha \u)= \bF(\u)+(\alpha -1)M\u
+o(\alpha-1)=\lambda \u + (\alpha-1)M\u+o(\alpha-1)$, and so $M\u=\lambda \u$,
which shows that $\lambda$ is the Perron root of $M$. Since, as observed
above, $M$ is primitive,
the eigenvalues $\lambda_1,\ldots,\lambda_n$ of $M$ can be ordered
in such a way that $\lambda=\lambda_1> r=|\lambda_2|\geq |\lambda_3|
\geq \cdots \geq |\lambda_n|$.
Consider now $\bG(\x):=(\psi\trans\bF(\x))^{-1}\bF(\x)$. An elementary computation
shows that
\[
\bG'(\u)= \lambda^{-1}(M-\u\psi\trans M) \enspace .
\]
We claim that the spectral radius $\rho(Q)$ of the matrix $Q:=M-\u\psi\trans M$
is equal to $|\lambda_2|$.
Assume first that all the eigenvalues
of $M$ are distinct and that no eigenvalue of $M$ is $0$.
Since $\psi\trans \u=1$, we have $Q\u=M\u-\u\psi\trans M\u=\lambda \u-\lambda \u =0$,
so $0$ is an eigenvalue of $Q$. Moreover, denoting
by $\varphi_j\trans$ a left eigenvector of $M$ for the
eigenvalue $\lambda_j$, with $2\leq j\leq n$,
we get
$\lambda_j\varphi_j\trans\u=\varphi_j\trans M\u=\lambda \varphi_j\trans\u$, and
since $\lambda_j\neq \lambda$, $\varphi_j\trans\u=0$. It follows
that $\lambda_j\varphi_j\trans= \varphi_j\trans Q$.
Hence, the eigenvalues of the matrix $Q$ are precisely
$0,\lambda_2,\ldots,\lambda_n$. Thus, $\rho(Q)=|\lambda_2|$.
Assume now that $M$ does not satisfy the above assumptions.
Let $D=\diag(d_1,\ldots,d_n)$ be a diagonal matrix such that $d_1>d_2>\ldots>d_n>0$.
Consider the matrix $M(t)=M+tD$.  For $t\gg 1$ the matrix $t^{-1}M(t)$ has eigenvalues
$d_i+O(\frac{1}{t})$ for $i=1,\ldots,n$.  Hence $\det M(t)$ vanishes at a finite number of $t$'s.
The discriminant of the polynomial $\det(xI -M(t))$, as a polynomial in $x$, is a nontrivial polynomial of
$t$.  Hence $M(t)$ has simple eigenvalues for all but a finite number of $t$'s.
Thus there exists $\varepsilon>0$ so that for each $t\in(0,\varepsilon)$ $M(t)$ has simple pairwise
distinct eigenvalues which are nonzero. Assume that $t\in(0,\varepsilon)$.
Clearly $M(t)$  is primitive, with the eigenvalues $\lambda_1(t)>|\lambda_2(t)|\ge
\ldots\ge |\lambda_n(t)|$.  Let $M(t)\u(t)=\lambda_1\u(t)$, where $\u(t)>\0$ and $\psi\trans\u(t)=1$.  Hence the spectral radius of $Q(t):=M(t)-\u(t)\psi\trans M(t)$ is $|\lambda_2(t)|$.
As the spectral radius of $Q(t)$ is continuous in its entries, by letting $t\searrow 0$ we deduce that $\rho(Q)=r$.
Thus the spectral radius of $\bG'(\u)$, denoted by $\rho(\bG'(\u))$, is equal to $\lambda^{-1}r$.
Since
\[
\x^{(k+1)}-\u=\bG(\x^{(k)})-\bG(\u)=\bG'(\u)(\x^{(k)}-\u)+o(\|\x^{(k)}-\u\|) \enspace.
\]
A classical result (Lemma~$5.6.10$ in~\cite{HoJo}) shows
that for all $\beta>\rho(\bG'(\u))$, there exists a norm $\|\cdot\|$
on $\mathbb{R}^n$ for which $\bG'(\u)$ is a contraction of rate $\beta$.
By the previous corollary, $\x^{(k)}$ tends to $\u$ as $k\to\infty$.
Hence, for all $\epsilon>0$, we can find an index $k_0$ such that
\[
\|\x^{(k+1)}-\u\| \leq \beta \|\x^{(k)}-\u\|+\epsilon \|\x^{(k)}-\u\| \enspace,
\]
holds for all $k\geq k_0$, and so
\[
\limsup_{k\to\infty} \|\x^{(k)}-\u\|^{1/k}\leq \beta+\epsilon
\enspace .
\]
Taking the infimum over $\beta>\rho(\bG'(\u))$ and $\epsilon>0$, we
deduce that
\[
\limsup_{k\to\infty} \|\x^{(k)}-\u\|^{1/k}\leq \rho(\bG'(\u))
\enspace .
\]
\qed
\end{proof}

We note that the previous corollaries apply in particular to the polynomial
map appearing in our initial problem~\eqref{critveceq}.


 \section{Examples and remarks}\label{sec-ex}
 We first give numerical examples showing that the conclusion
 of Theorem \ref{maintheo} no longer holds for $p<d$.
 Consider first the positive tensor $\cF_1 \in\R^{2\times 2\times
 2}$ with entries
 \begin{equation}
 f_{1,1,1} = f_{2,2,2}=a>0,\; {\rm otherwise,} \ f_{i,j,k}=b>0
 \end{equation}
 So the trilinear form is
 $$f(\x_1,\x_2,\x_3)=b(x_{1,1}+x_{2,1})(x_{1,2}+x_{2,2})(x_{1,3}+x_{2,3})+
 (a-b)(x_{1,1}x_{1,2}x_{1,3}+x_{2,1}x_{2,2}x_{2,3}).
 $$
 Clearly, the system (\ref{critveceq}) for $p_1=p_2=p_3=p>1$
 has a positive solution $\x_1 = \x_2 = \x_3 = (0.5^{1/p}, 0.5^{1/p}
 )\trans$.
 Let
 \begin{equation}
 f_{1,1,1} = f_{2,2,2}=a=1.2; \ {\rm otherwise,} \ f_{i,j,k}=b=0.2.
 \end{equation}
 For this tensor, the system (\ref{critveceq}) has a unique solution
 $\x_1 = \x_2 = \x_3 = (0.5^{1/p}, 0.5^{1/p} )\trans $ for $p \geq 3$.
 However, for $p=2\;(<d=3)$ (\ref{critveceq}), in addition to the above
 positive solution, has two other positive solutions
 $$
 \x_1 = \x_2 = \x_3 \approx (0.9342, 0.3568)\trans
 $$
 and
 $$
 \x_1 = \x_2 = \x_3 \approx (0.3568, 0.9342)\trans.
 $$

 There are weakly irreducible tensors for which the conclusion of Theorem \ref{maintheo} can fail for
 $p$ very close to $d$. As an example, we  consider the positive tensor
 $\cF_2 \in\R^{2\times 2\times
 2}$ with entries
 \begin{equation}
 f_{1,1,1} = f_{2,2,2}=a=1.001; \ {\rm otherwise,} \ f_{i,j,k}=b=0.001.
 \end{equation}
  For tensor $\cF_2$, the system (\ref{critveceq})  has additional two positive solutions:
 $$
 \x_1 = \x_2 = \x_3 \approx (0.9667, 0.4570)\trans
 $$
 and
 $$
 \x_1 = \x_2 = \x_3 \approx (0.4570, 0.9667)\trans
 $$
 when $p=2.99$. \\

 We now show that the results in \cite{CPZ08} do not apply in
 our case for $d>2, p=d$.  For simplicity of the discussion we consider the
 case $d=3$.  The homogeneous eigenvalue problem studied in
 \cite{CPZ08} is for the nonnegative tensor $\cC=[c_{i,j,k}]\in
 \R_+^{n\times n \times n}$.
 It is of the form
 \begin{equation}\label{cpz08pr}
 \sum_{j=k=1}^{n} c_{i,j,k} y_jy_k= \lambda y_i^2,
 \;i\in\setint[n].
 \end{equation}

 The Perron-Frobenius theorem in \cite{CPZ08} is proved for
 irreducible tensors.
 We now show that the induced tensor $\cC$ by the tensor $\cF$
 is reducible, i.e. not irreducible.  To this end let $n=m_1+m_2+m_3$ and
 define
 $$\y=(x_{1,1},\ldots,x_{m_1,1},x_{1,2},\ldots,x_{m_2,2},
 x_{1,3},\ldots,x_{m_3,3})\trans\in\R^n.$$
 Then the system (\ref{critveceq}) for $d=3, p=3$
 can be written as the system (\ref{cpz08pr}).
 We claim that $\cC$ is always reducible.
 Indeed, choose $\y$ corresponding to
 $\x_1=\x_2=\0, \x_3=\1$.  Clearly, the left-hand side of
 (\ref{critveceq}) is zero for all equations.
 Hence the left-hand side of (\ref{cpz08pr}) is zero for all
 $i\in\setint[n]$.
Hence $\cC$ is reducible.

 We close this section with a variation on the classical
 Perron-Frobenius theorem on bilinear form $\x\trans A\y$,
 where $A\in \R_+^{m\times n}$ is a nonnegative matrix.  If the
 bipartite graph $G(A)$ induces a connected bipartite graph
 then the largest singular value of $A$, equal to $\|A\|$,
 is simple, with the unique nonnegative left and right singular
 vectors $\xib,\etab$
 of $A$, corresponding $\|A\|$, of length one which are
 positive.  This is the classical Perron-Frobenius theorem.
 Theorem \ref{maintheo} claims that if the induced bipartite graph $G(A)$
 is connected then the classical Perron-Frobenius theorem holds
 for any $p_1=p_2 \geq d = 2$.  However, in the case
 $p_1=p_2<d=2$ the Perron-Frobenius theorem may fail
 as in the case $d=3$.  Consider the following example.
 $$A =\left[\begin{array}{cccc}

    1.0000  &  0.2000 &   0.2000  &  0.2000\\
    0.2000  &  1.0000 &   0.2000  &  0.2000 \\
    0.2000  &  0.2000 &   1.0000  &  0.2000
    \end{array}\right]
 $$
 When $p_1=p_2=p=1.5<d=2$, the system (\ref{critveceq}) has three solutions:
 \begin{eqnarray*}
 \x=(0.0893, 0.9641, 0.0893)\trans, \;
 \y= (0.0863, 0.9583, 0.0863, 0.0501)\trans,\\
 \x=  (0.0893,    0.0893,    0.9641)\trans,\;
 \y= ( 0.0863,    0.0863,    0.9583,    0.0501)\trans,\\
 \x= (0.9641,   0.0893,    0.0893)\trans,\;
 \y= (0.9583,   0.0863,    0.0863,    0.0501)\trans.
 \end{eqnarray*}

 For the same matrix, if $p_1=1.2$ and $p_2=2.5$, the system (1.2) also has
 three  positive solutions.

 \bibliographystyle{plain}

\end{document}